\documentclass[12pt,reqno]{amsart}

\usepackage{amsthm, mathrsfs,amssymb,amsmath}
\usepackage{enumerate}
\usepackage{hyperref}
\usepackage{xcolor}
\hypersetup{
	colorlinks,
	linkcolor={red!50!black},
	citecolor={green!50!black},
	urlcolor={red!80!black}
}
\makeatletter
\@namedef{subjclassname@2010}{%
	\textup{2010} Mathematics Subject Classification}
\makeatother

\allowdisplaybreaks

\newtheorem{thm}{Theorem}[section]
\newtheorem{lem}[thm]{Lemma}

\newtheorem{cor}[thm]{Corollary}

\theoremstyle{definition}

\theoremstyle{remark}
\newtheorem{rem}{Remark}

\def\R2n{{\mathbb R}^{2n}}

\def\R^2{{\mathbb R}^2}
\def\R2n{{\mathbb R}^{2n}}
\def\R{{\mathbb R}}

\title[ Boundedness of $(k, a)$-Fourier multipliers]{$L^p$-$L^q$ boundedness of $(k, a)$-Fourier multipliers with applications to Nonlinear equations}
\author{Vishvesh Kumar}
\address{Vishvesh Kumar \endgraf
	Department of Mathematics: Analysis, Logic and Discrete Mathematics
	\endgraf
	Ghent University, Belgium}
\endgraf
\email{vishveshmishra@gmail.com}
\author[Michael Ruzhansky]{Michael Ruzhansky}
\address{
	Michael Ruzhansky:
	\endgraf
	Department of Mathematics: Analysis, Logic and Discrete Mathematics
	\endgraf
	Ghent University, Belgium
	\endgraf
	and
	\endgraf
	School of Mathematical Sciences
	\endgraf
	Queen Mary University of London
	\endgraf
	United Kingdom
	\endgraf
	{\it E-mail address} {\rm michael.ruzhansky@ugent.be}
}

\begin{document}
	\begin{abstract} The $(k,a)$-generalised Fourier transform  is the unitary operator defined using the $a$-deformed Dunkl harmonic oscillator.
	The main aim of this paper is to prove $L^p$-$L^q$ boundedness of $(k, a)$-generalised Fourier multipliers. To show the boundedness we first establish Paley inequality and Hausdorff-Young-Paley inequality for $(k, a)$-generalised Fourier transform. We also demonstrate applications of obtained results to study the well-posedness of nonlinear partial differential equations. 
	\end{abstract}
	\maketitle
	
	\section{Introduction and Basics on $(k,a)$-generalised Fourier transform}
	
In his seminal paper \cite{Hormander1960}, H\"ormander initiated the study of boundedness of the translation invariant operators on $\mathbb{R}^N$. The translation invariant operators on $\mathbb{R}^N$ can be characterised using the classical Euclidean Fourier transform on $\mathbb{R}^N$ and therefore they are also known as Fourier multipliers. The boundedness of Fourier multipliers is useful to solve problems in the area of mathematical analysis, in particular, in PDEs. H\"ormander \cite{Hormander1960} established the $L^p$- boundedness and $L^p$-$L^q$ boundedness of Fourier multipliers on $\mathbb{R}^N$. After that, $L^p$-boundedness of Fourier multipliers has been investigated by several researchers in many different setting, we cite here \cite{ Hormander1960, Anker, Stein, Ruzwirth, BCC, BX99, DPW, Soltani, Wro, GStempak, DHejna} to mention a few of them. In particular, $L^p$-boundedness of multipliers was established in \cite{Soltani} for the one dimensional Dunkl transform and very recently in \cite{DHejna} in the multidimensional setting. Recently, the researchers have turned their attention to establish the boundedness of $L^p$-$L^q$ multipliers for the range $1<p \leq 2\leq   q< \infty,$ see \cite{ARN1, AR, Amri, CKNR, CK, Cow1, KR}. Precisely,  the second author and his collaborators started investigating the H\"ormander $L^p$-$L^q$ Fourier multipliers theorem and its different consequences for locally compact groups and on homogeneous manifolds. Such analysis  includes the Hardy-Littlewood inequality, spectral multipliers theorems and applications to PDEs \cite{ARN1, AR, ARN, KR}. In \cite{CK}, similar results have been proved for the eigenfunction expansions of anharmonic oscillators and extended to the more general setting of bi-orthogonal expansions in \cite{CKNR}. Ben Sa\"id et al. \cite{BKO, BKO2} introduced $(k,a)$-generalised Fourier transform. It generalises many important integral transforms including Fourier transform and Dunkl transform on the Euclidean spaces $\mathbb{R}^N$ \cite{Ben, BKO2}. Recently, there is a growing interest to develop the analysis related to the $(k,a)$-generalised Fourier transform. Notably, the uncertainty principles and Pitt inequalities \cite{GIT, John}, maximal function and translation operator \cite{Luc}, wavelets multipliers \cite{Hatem} and Hardy inequality \cite{Teng} were explored by many researchers. 
In this paper, we establish $L^p$-$L^q$ boundedness of $(k, a)$-Fourier multipliers  using the $(k,a)$-generalised Fourier transform. The proof of the main result hinges upon the Paley inequality and Hausdorff-Young-Palay inequality for $(k,a)$-generalised Fourier transform obtained by using the Hausdorff-Young inequality established in \cite{John, GIT}.  

To describe our main result let us recall the classical H\"ormander Fourier multipliers theorem settled in \cite{Hormander1960}: For $1<p\leq 2 \leq q <\infty,$  the Fourier multiplier $T_m: \mathcal{S}(\mathbb{R}^N) \rightarrow \mathcal{S}'(\mathbb{R}^N)$ associated with symbol $m:\mathbb{R}^N \rightarrow \mathbb{C}$  defined by $\mathcal{F}(T_m f)(\xi)= m(\xi) \mathcal{F}(f)(\xi)$ for $\xi \in \mathbb{R}^N,$ has a bounded extension from $L^p(\mathbb{R}^N)$ to $L^q(\mathbb{R}^N)$ provided that the symbol $m$ satisfies the condition
\begin{equation}
    |\{\xi \in \mathbb{R}^N: |m(\xi)| \geq s\}| \leq \frac{1}{s^b}\,\,\,\textnormal{for all}\,\, s>0,
\end{equation} where $\frac{1}{b}=\frac{1}{p}-\frac{1}{q},$ and $\mathcal{F}$ denotes the Euclidean Fourier transform of $f$ defined as 
$$\mathcal{F}(f)(\xi):=(2\pi)^{-\frac{N}{2}}\int_{\mathbb{R}^N} f(x)\, e^{-i \langle x, \xi \rangle} \, dx,\quad \xi \in \mathbb{R}^N.$$ Here  $\langle x, \xi \rangle$ denotes the standard Euclidean inner product of two vectors $x$ and $\xi$ in $\mathbb{R}^N$ and $\|x\|$ will denote the Euclidean norm on $\mathbb{R}^N.$ 
The Euclidean Fourier transform $\mathcal{F}$ on $\mathbb{R}^N$ can be described using the spectral information of the harmonic oscillator $\Delta_{\mathbb{R}^N}-\|x\|^2,$ where $\Delta_{\mathbb{R}^N}$ is the Laplacian on $\mathbb{R}^N.$ In fact, Howe \cite{Howe} found the following description of the Euclidean Fourier transform $\mathcal{F}:$
\begin{equation}
    \mathcal{F}:= \exp \Big( \frac{i\pi N}{4} \Big) \exp \Big( \frac{i \pi}{4} (\Delta_{\mathbb{R}^N}-\|x\|^2) \Big).
\end{equation}
This description has been proved to be useful to define generalisations of the Fourier transform such as Clifford algebra-valued Fourier transform and fractional Fourier transform. These constructions have been explained in an excellent overview article \cite{DeBie}. On the other hand, Dunkl  \cite{Dunkl, Dunkl2} presented a generalisation of the Euclidean Fourier transform and Euclidean Laplacian on $\mathbb{R}^N$, which is now known as Dunkl transform (see \cite{Dejeu}) and Dunkl Laplacian, and are usually denoted by $\mathcal{F}_{k}$ and $\Delta_{k},$ respectively,  using the root system $\mathcal{R} \subset \mathbb{R}^N,$ a reflection group $\mathfrak{G} \subset O(N, \mathbb{R})$  generated by the root reflections $r_\alpha,$ $\alpha \in \mathcal{R},$ and a multiplicity function $k: \mathcal{R} \rightarrow \mathbb{R}_+$ such that $k$ is $\mathfrak{G}$-invariant. We set $k(\alpha)=k_\alpha,$ $\langle k \rangle=\frac{1}{2}\sum_{\alpha \in \mathcal{R}} k_\alpha,$ $v_k(x)=\prod_{\alpha \in \mathcal{R}} |\langle \alpha, x \rangle|^{ k_\alpha},$ $v_{k, a}(x):= \|x\|^{2-a} v_k(x).$ Define $L^p_{k,a}(\mathbb{R}^N):=L^p(\mathbb{R}^N, v_{k,a}dx)$ and $d\mu_{k,a}(x)=v_{k,a} dx.$ 

To describe the Dunkl Laplacian, let us define the first order Dunkl operator for $\xi \in \mathbb{R}^N$ and for a fixed multiplicity function $k$  by 
$$T_\xi(k)f(x)= \partial_\xi f(x)+\sum_{\alpha \in \mathcal{R}_+} k_\alpha \, \langle \alpha, \xi \rangle \frac{f(x)-f(r_{\alpha}x)}{\langle \alpha, x \rangle},\quad f \in C^1(\mathbb{R}^N),$$ where $\partial_\xi$ is the direction derivation in the direction of $\xi$ and $\mathcal{R}_+$ denotes the positive root subsystem.
Let us fix an orthonormal basis $\{\xi_1, \xi_2, \ldots, \xi_N\}$ for the inner product space $(\mathbb{R}^N, \langle\cdot, \cdot \rangle)$ and write $T_{\xi_j}(k)$ as $T_j(k)$ for $j \in \{1, 2, \ldots, N\}.$ Then the Dunkl Laplacian is defined by $\Delta_k=\sum_{j=1}^N T_j(k)^2.$ The Dunkl Laplacian has explicit form and also plays a very important role in the Dunkl analysis (see \cite{Anker2, Rosler} for more details and related analysis).
When the multiplicity function is trivial (i.e., $k \equiv 0$) then $F_{k}$ and $\Delta_k$ turn out be to just the Euclidean Fourier transform $\mathcal{F}$ and the Euclidean Laplacian $\Delta_{\mathbb{R}^N},$ respectively. Using the Dunkl Laplacian one can define the Dunkl harmonic oscillator (or Dunkl-Hermite operator) as $\Delta_k-\|x\|^2$.  
Ben Sa\"id et al. \cite{BKO} considered the $a$-deformed Dunkl harmonic oscillator given by 
$$\Delta_{k, a}:= \|x\|^{2-a} \Delta_k-\|x\|^a,\quad a>0.$$ By making use of this $a$-deformed Dunkl harmonic oscillator $\Delta_{k,a}$, they introduced a two parameters unitary operator, {\it $(k, a)$-generalised Fourier transform}, $\mathcal{F}_{k,a}$ on $L^2_{k,a}(\mathbb{R}^N),$ by
\begin{equation} \label{kadunkl}
    \mathcal{F}_{k,a}:= \exp\left [\frac{i \pi}{2}\left(\frac{1}{a} (2 \langle k \rangle+n+a-2) \right)\right] \exp \left[\frac{i \pi}{2a} (\Delta_{k, a}) \right].
\end{equation}

The $(k,a)$-generalised Fourier transform $\mathcal{F}_{k,a}$ includes some prominent transforms on the Euclidean space $\mathbb{R}^N:$
\begin{itemize}
    \item For $a=2$ and $k>0$, $\mathcal{F}_{k,a}$  is the Dunkl transform \cite{Dejeu}.
    \item For $a=2$ and $k\equiv 0,$ $\mathcal{F}_{k,a}$ is the Euclidean Fourier transform \cite{Howe}.
    \item For $a=1$ and $k \equiv 0,$ $\mathcal{F}_{k,a}$ is the Hankel transform appearing as the unitary inversion operator of the Schr\"odinger model of the minimal representation of the group $O(N+1, 2)$ (see \cite{KobMano, KobMano2, KobMano3}).
\end{itemize}

For $a>0$ and $a+2\langle k \rangle+N-2>0,$ the $(k,a)$-generalised Fourier transform $\mathcal{F}_{k,a}$ is a bijective linear operator such that 
\begin{equation}
    \|\mathcal{F}_{k,a}(f)\|_{L_{k,a}^2(\mathbb{R}^N)}=\|f\|_{L_{k,a}^2(\mathbb{R}^N)}.
\end{equation}
By the Schwartz kernel theorem there exists a distribution kernel $B_{k, a}(\xi, x)$ such that 
$$\mathcal{F}_{k,a}f(\xi)=c_{k,a} \int_{\mathbb{R}^N} B_{k,a}(\xi,x) \,f(x)\, d\mu_{k,a}(x)$$
with a symmetric kernel $B_{k,a}(\xi, x)$ (\cite{BKO}).

The next lemma, which is a corrected version of \cite[Lemma 2.8]{John} in view of \cite[ Section 6]{GIT},  presents  some conditions on $N, k$, and $a$ such that kernel $B_{k,a}(\xi, x)$ is uniformly bounded  (see also \cite[Theorem 5.11]{BKO} and \cite[Theorem 9]{DDL}).
\begin{lem}  \label{standing}Assume $N \geq 1,\, k \geq 0,\, a+2 \langle k \rangle+N-2 >0,$ and that exactly one of the following additional assumption holds: 
   \begin{itemize}
       \item[(i)] $N=1$  and  $a>0$;
       \item[(ii)] $a=1$ and $2 \langle k \rangle+N-2 \geq 0$
       \item[(iii)] $a=2;$
       \item[(iii)] $k=0$ and  $a=\frac{2}{m}$  for some\,\,$m \in \mathbb{N}.$
   \end{itemize}
   Then $B_{k, a}$ is uniformly bounded, that is, $|B_{k, a}(\xi, x)| \leq M$ for all $x, \xi \in \mathbb{R}^N,$ where $M$ is a finite constant that depends only on $N,\,k,\,$ and $a.$ 
	\end{lem}
	
	The following result is the Hausdorff -Young inequality for $(k, a)$- generalised Fourier transform.
	\begin{thm} \cite[Proposition 2.9]{John} \label{HY} Assume that $N,\, k,\,$ and $a$ satisfy the assumption of Lemma \ref{standing}. For $1 \leq p \leq 2,$ fix $p'=\frac{p}{p-1}.$ Then for $f \in L^p_{k,a}(\mathbb{R}^N)$ we have 
	\begin{equation} \label{HYineq}
	     \|\mathcal{F}_{k, a} f\|_{L^{p'}_{k, a}(\mathbb{R}^N)} \leq C \|f\|_{L^p_{k, a}(\mathbb{R}^N)},
	\end{equation} 
	where $C=M^{2/p-1}.$
	\end{thm}
	It was conjectured by Gorbachev et al. \cite{GIT} that if $a+2\langle k \rangle+N-3\geq 0$ then the kernel satisfies $|B_{k,a}(\xi, x)| \leq B_{k,a}(0, x)=1 $ for all $x, \xi \in \mathbb{R}^N.$ So in this case, the constant $C$ in Hausdorff-Young inequality \eqref{HYineq} becomes $1.$
	
{\it	From this point onward, we always assume that $N, k$ and $a$ either satisfy assumptions  of Lemma \ref{standing} with $N \geq 1, k \geq 0$ and $a+2\langle k \rangle+N-2>0,$ or, $a+2\langle k \rangle+N-3\geq 0$ without mentioning it explicitly.} In fact, our results will hold if we assume that $N\geq 1, k\geq 0$ and $a>0$ are such that $a+2\langle k \rangle+N-2>0$ and  the distribution kernel $B_{k,a}$ is uniformly bounded on $\mathbb{R}^N.$
	
	With having all the basics of $(k,a)$-generalised Fourier transform  we are now in a position to state our results. The main result of this paper is the following theorem $L^p$-$L^q$ boundedness of $(k,a)$-Fourier multipliers $A$ for the range $1<p \leq 2 \leq q<\infty.$ Indeed, we have 
	$$\|A\|_{L^p_{k, a}(\mathbb{R}^N) \rightarrow L^q_{k, a}(\mathbb{R}^N)}\lesssim \sup_{s>0} s \left[ \int_{\{ \xi \in \mathbb{R}^N: |h(\xi)|\geq s\}} d\mu_{k, a}(\xi) \right]^{\frac{1}{p}-\frac{1}{q}},$$ where $h$ is the symbol of the $(k,a)$-Fourier multiplier $A$, this means that, $\mathcal{F}_{k,a}(Af)(\xi)= h(\xi) \mathcal{F}_{k,a}f(\xi)$ for $\xi \in \mathbb{R}^N$ and for $f$ in a suitable function space. The main tool to establish this result is the following Hausdorff-Young Paley inequality for $(k,a)$-generalised Fourier transform: For $1<p\leq 2,$   $1<p \leq b \leq p' < \infty,$ where $p'= \frac{p}{p-1}$ and for a positive function $\psi$ defined on $\mathbb{R}^N$ we have  
\begin{equation} 
    \left( \int_{\mathbb{R}^N}  \left( |\mathcal{F}_{k, a}f(\xi)| \psi(\xi)^{\frac{1}{b}-\frac{1}{p'}} \right)^b d\mu_{k, a}(\xi)  \right)^{\frac{1}{b}} \lesssim \Big( \sup_{t>0} t \int_{\underset{\psi(\xi)\geq t}{\xi \in \mathbb{R}^N}} d\mu_{k, a}(\xi) \Big)^{\frac{1}{b}-\frac{1}{p'}} \|f\|_{L^p_{k, a}(\mathbb{R}^N)}.
	\end{equation}
	Next, we will present applications of our main results in the context of well-posedness of nonlinear abstract Cauchy problems in the space $L^\infty(0, T, L^2_{k,a}(\mathbb{R}^N)).$ First, we consider the heat equation 
 \begin{equation} \label{Heatinto} 
     u_t-|Bu(t)|^p=0,\quad u(0)=u_0,
 \end{equation}
 where $B$ is a linear operator on $L^2_{k, a}(\mathbb{R}^N)$ and $1<p<\infty.$ We study local well-posedness of the  heat equation \eqref{Heatinto} above. Secondly, we consider  
 the initial value problem  for the nonlinear wave equation  
\begin{align}\label{E-WNLEint}
u_{tt}(t) - b(t)|Bu(t)|^{p} = 0,
\end{align} with the initial condition $
u(0)=u_0, \,\,\, u_t(0)=u_1, $
where $b$ is a positive bounded function depending only on time, $B$ is a linear operator in $L^2_{k,a}(\mathbb{R}^N)$ and $1< p<\infty.$ We explore the global and local well-posedness of \eqref{E-WNLEint} under some condition on function $b.$

We organise the paper in following way: In the next section we will state and present the proof of Paley inequality and Hausdorff-Young-Paley inequality.   Then, we give the proof of our main result concerning the $L^p$-$L^q$ boundedness of $(k,a)$-Fourier multipliers and its consequences. In the last section, the applications of the results obtained in previous section will be discussed.

	\section{Main results}
	Throughout the paper, we shall use the notation $A \lesssim B$ to indicate $A\leq cB $ for a suitable constant $c >0$.
	In this section, we will present our main results. In the proofs we follow the ideas in the papers \cite{ARN, AR}.  The first result is the Paley inequality for the  $(k,a)$-generalised Fourier transform.
\begin{thm}\label{Paley} 
	      Suppose that $\psi$ is a positive function  on $\mathbb{R}^N$  satisfying the condition 
	    \begin{equation}
	        M_\psi := \sup_{t>0} t \int_{\underset{\psi(\xi)\geq t}{\xi \in \mathbb{R}^N}} d\mu_{k, a}(\xi) <\infty.
	        \end{equation}
	        Then for  $f \in L^p_{k,a}(\mathbb{R}^N),$ $1<p\leq 2,$ we have 
	        \begin{align} \label{Paleyin}
	            \left( \int_{\mathbb{R}^N} |\mathcal{F}_{k, a}(\xi)|^p\, \psi(\xi)^{2-p} d\mu_{k, a}(\xi) \right)^{\frac{1}{p}} \lesssim M_{\psi}^{\frac{2-p}{p}}\, \|f\|_{L^p_{k, a}(\mathbb{R}^N)}.
	        \end{align}
	\end{thm}
	\begin{proof}  Let us consider a measure $\nu_{k, a}$ on $\mathbb{R}^N$ given by \begin{equation} \label{nu}
	    \nu_{k, a}(\xi)= \psi(\xi)^2 d\mu_{k, a}(\xi).
	\end{equation}
	We define the corresponding $L^p(\mathbb{R}^N, \nu_{k, a})$- space, $1 \leq p<\infty,$ as the space of all complex-valued function $f$ defined by $\mathbb{R}^N$ such that 
	$$ \|f\|_{L^p(\mathbb{R}^N, \nu_{k, a})}:= \left( \int_{\mathbb{R}^N} |f(\xi)|^p \, \psi(\xi)^2 d\mu_{k, a}(\xi) \right)^{\frac{1}{p}}<\infty.$$
	 We will show that the sublinear operator $T: L^p_{k, a}(\mathbb{R}^N) \rightarrow L^p(\mathbb{R}^N, \nu_{k, a})$ defined by 
	$$ Tf(\xi):= \frac{|\mathcal{F}_{k,a}f(\xi)|}{\psi(\xi)},\,\,\xi \in \mathbb{R}^N$$
is well-defined and bounded from $L^p_{k, a}(\mathbb{R}^N)$ to $L^p(\mathbb{R}^N, \nu_{k, a})$ for any $1 < p \leq 2.$	
In other words, we claim the following estimate:
\begin{equation} \label{vis10paley}
    \|Tf\|_{L^p(\mathbb{R}^N, \nu_{k,a})}= \left( \int_{\mathbb{R}^N} \frac{|\mathcal{F}_{k, a}f(\xi)|}{\psi(\xi)^p}\,\psi(\xi)^2 d\mu_{k, a}(\xi) \right)^{\frac{1}{p}} \lesssim M_{\psi}^{\frac{2-p}{p}} \|f\|_{L^p_{k,a}(\mathbb{R}^N)},
\end{equation}
which will give us the required inequality \eqref{Paleyin} with $M_\psi := \sup_{t>0} t \int_{\underset{\psi(\xi) \geq t}{\xi \in \mathbb{R}^N}} d\mu_{k,a}(\xi).$ 
	We will show that $T$ is weak-type $(2,2)$ and weak-type $(1,1).$  More precisely, with the distribution function,  
$$\nu_{k, a}(y; Tf)= \int_{\underset{\frac{|\mathcal{F}_{k, a}f(\xi)|} {\psi(\xi)} \geq y}{\xi \in \mathbb{R}^N}} \psi(\xi)^2 d\mu_{k, a}(\xi), $$ where $\nu_{k, a}$ is given by formula \eqref{nu}, we show that 
\begin{equation} \label{vish5.4}
    \nu_{k, a}(y; Tf) \leq \left( \frac{M_2 \|f\|_{L^2_{k, a}(\mathbb{R}^N)}}{y} \right)^2 \,\,\,\,\,\,\text{with norm}\,\, M_2=1,
\end{equation}
\begin{equation} \label{vish5.5}
    \nu_{k, a}(y; Tf) \leq \frac{M_1 \|f\|_{L^1_{k, a}(\mathbb{R}^N)}}{y}\,\,\,\,\,\,\text{with norm}\,\, M_1=M_\psi.
\end{equation} 
	Then the estimate \eqref{vis10paley} follows from the Marcinkiewicz interpolation Theorem. Now, to show \eqref{vish5.4}, using the Plancherel identity we get
	\begin{align*}
    y^2 \nu_{k, a}(y; Tf)&\leq \sup_{y>0}y^2 \nu_{k, a}(y; Tf)=: \|Tf\|^2_{L^{2, \infty}(\mathbb{R}^N, \nu_{k, a})}  \leq \|Tf\|^2_{L^2(\mathbb{R}^N, \nu_{k, a})} \\&= \int_{\mathbb{R}^N} \left( \frac{|\mathcal{F}_{k, a}(\xi)|}{\psi(\xi)} \right)^2 \psi(\xi)^2 d\mu_{k, a}(\xi) \\&= \int_{\mathbb{R}^N} |\mathcal{F}_{k, a}(\xi)|^2\, d\mu_{k, a}(\xi)  = \|f\|_{L^2_{k, a}(\mathbb{R}^N)}^2.  \end{align*}
Thus, $T$ is type $(2,2)$ with norm $M_2 \leq 1.$ Further, we show that $T$ is of weak type $(1,1)$ with norm $M_1=M_\psi$; more precisely, we show that 
\begin{align} \label{11weak}
    \nu_{k, a} \left\{ \xi \in \mathbb{R}^N: \frac{|\mathcal{F}_{k, a}(\xi)|}{\psi(\xi)}>y \right\} \lesssim M_\psi \frac{\|f\|_{L^1_{k, a}(\mathbb{R}^N)}}{y}.
\end{align}
The left hand side of \eqref{11weak} is an integral $ \int \psi(\xi)^2 d\mu_{k, a}(\xi)$ taken over all those $\xi \in \mathbb{R}^N$ for which $\frac{|\mathcal{F}_{k, a}f(\xi)|}{\psi(\xi)}>y.$
Since $|\mathcal{F}_{k, a}f(\xi)| \lesssim \|f\|_{L^1_{k, a}(\mathbb{R}^N)}$ for all $\xi \in \mathbb{R}^N$ we have
$$\left\{ \xi \in \mathbb{R}^N: \frac{|\mathcal{F}_{k, a}f(\xi)|}{\psi(\xi)}>y \right\} \subset \left\{ \xi \in \mathbb{R}^N: \frac{\|f\|_{L^1_{k, a}(\mathbb{R}^N)}}{\psi(\xi)} \gtrsim y \right\},$$ 
 for any $y>0$ and, therefore,
$$\nu_{k, a} \left\{ \xi \in \mathbb{R}^N: \frac{|\mathcal{F}_{k, a}(\xi)|}{\psi(\xi)}>y \right\} \leq \nu_{k, a} \left\{ \xi \in \mathbb{R}^N: \frac{\|f\|_{L^1_{k,a}(\mathbb{R}^N)}}{\psi(\xi)}\gtrsim y \right\}.$$
Now by setting $w:=\frac{\|f\|_{L^1_{k,a}(\mathbb{R}^N)}}{y},$ we have 
\begin{align}
    \nu_{k, a} \left\{ \xi \in \mathbb{R}^N: \frac{\|f\|_{L^1_{k, a}(\mathbb{R}^N)}}{\psi(\xi)} \gtrsim y \right\} \leq  \int_{\overset{\xi \in \mathbb{R}^N}{\psi(\xi) \lesssim w} } \psi(\xi)^2\, d\mu_{k, a}(\xi).
\end{align}
Now we claim that 
\begin{align} \label{claim}
    \int_{\overset{\xi \in \mathbb{R}^N}{\psi(\xi) \lesssim w} } \psi(\xi)^2\, d\mu_{k, a}(\xi) \lesssim M_\psi w.
\end{align}
Indeed, first we notice that 
\begin{align*}
    \int_{\overset{\xi \in \mathbb{R}^N}{\psi(\xi) \lesssim w} } \psi(\xi)^2\, d\mu_{k, a}(\xi) = \int_{\overset{\xi \in \mathbb{R}^N}{\psi(\xi) \leq c w} } \, d\mu_{k, a}(\xi) \int_0^{\psi(\xi)^2} d\tau, 
\end{align*} for some $c>0.$
By interchanging the order of integration we get 
\begin{align*}
    \int_{\overset{\xi \in \mathbb{R}^N}{\psi(\xi) \leq c w} } \,d\mu_{k, a}(\xi) \int_0^{\psi(\xi)^2} d\tau = \int_{0}^{c^2w^2} d\tau \int_{\underset{\tau^{\frac{1}{2}} \leq \psi(\xi) \leq c w}{\xi \in \mathbb{R}^N}} d\mu_{k, a}(\xi).
\end{align*}
Further, by making substitution  $\tau= t^2,$ it gives 
\begin{align*}
    \int_{0}^{c^2w^2} d\tau \int_{\underset{\tau^{\frac{1}{2}} \leq \psi(\xi) \leq c w}{\lambda \in \mathbb{R}^N}} d\mu_{k, a}(\xi) &= 2 \int_0^{cw} t\, dt \int_{\underset{t \leq \psi(\xi) \leq c w}{\xi \in \mathbb{R}^N}} d\mu_{k, a}(\xi) \\&\lesssim  \int_0^{cw} t\, dt \int_{\underset{t \leq \psi(\xi) }{\xi \in \mathbb{R}^N}} d\mu_{k, a}(\xi).
\end{align*}
Since 
$$ t \int_{\underset{t \leq \psi(\xi) }{\xi \in \mathbb{R}^N}} d\mu_{k, a}(\xi) \leq \sup_{t>0} t \int_{\underset{t \leq \psi(\xi) }{\xi \in \mathbb{R}^N}} d\mu_{k, a}(\xi) = M_\psi $$ is finite by assumption $M_\psi<\infty,$ we have 
\begin{align*}
     \int_0^w t\, dt \int_{\underset{t \leq \psi(\xi) }{\xi \in \mathbb{R}^N}} d\mu_{k, a}(\xi) \lesssim M_\psi w.
\end{align*}
This establishes our claim \eqref{claim} and eventually proves \eqref{11weak}. Therefore, we have proved \eqref{vish5.4} and \eqref{vish5.5}. Then  by using the Marcinkiewicz interpolation theorem with $p_1=1$ and $p_2=2$ and $\frac{1}{p}= 1-\theta+\frac{\theta}{2}$ we now obtain
$$\left( \int_{\mathbb{R}_+} \left(\frac{|\mathcal{F}_{k, a}f(\xi)|}{\psi(\xi)} \right)^p \psi(\xi)^2\, d\mu_{k, a}(\xi) \right)^{\frac{1}{p}}= \|Tf\|_{L^p(\mathbb{R}^N,\, \nu_{k, a})} \lesssim M_\psi^{\frac{2-p}{p}} \|f\|_{L^p_{k, a}(\mathbb{R}^N)}.$$
This completes the proof of the theorem. \end{proof}


Next we record the following  interpolation theorem from \cite{BL} for further use.  
\begin{thm}\label{interpolationoperator} Let $d\mu_0(x)= \omega_0(x) d\mu'(x)$ and $d\mu_1(x)= \omega_1(x) d\mu'(x).$ Suppose that $0<p_0, p_1< \infty.$  If a continuous linear operator $A$ admits bounded extensions, $A: L^p(Y,\mu)\rightarrow L^{p_0}(\omega_0) $ and $A: L^p(Y,\mu)\rightarrow L^{p_1}(\omega_1).$  Then, there exists a bounded extension $A: L^p(Y,\mu)\rightarrow L^{b}(\omega) $ of $A$, where  $0<\theta<1, \, \frac{1}{b}= \frac{1-\theta}{p_0}+\frac{\theta}{p_1}$ and 
 $\omega= \omega_0^{\frac{b(1-\theta)}{p_0}} \omega_1^{\frac{b\theta}{p_1}}.$
\end{thm} 

Now, we use the previous theorem to establish the  Hausdorff-Young-Paley inequality using the interpolation between Hausdorff-Young inequality and Paley inequality for $(k,a)$-generalised Fourier transform. 
\begin{thm} \label{HYP} Let $1<p\leq 2,$ and let   $1<p \leq b \leq p' < \infty,$ where $p'= \frac{p}{p-1}.$ If $\psi$ is a positive function on $\mathbb{R}^N$ such that 
 \begin{equation}
	        M_\psi := \sup_{t>0} t \int_{\underset{\psi(\xi)\geq t}{\lambda \in \mathbb{R}^N}} d\mu_{k, a}(\xi)
	        \end{equation}
is finite then, for every $f \in L^p_{k, a}(\mathbb{R}^N),$ 
 we have
\begin{equation} \label{Vish5.9}
    \left( \int_{\mathbb{R}^N}  \left( |\mathcal{F}_{k, a}f(\xi)| \psi(\xi)^{\frac{1}{b}-\frac{1}{p'}} \right)^b d\mu_{k, a}(\xi)  \right)^{\frac{1}{b}} \lesssim M_\varphi^{\frac{1}{b}-\frac{1}{p'}} \|f\|_{L^p_{k, a}(\mathbb{R}^N)}.
\end{equation}
\end{thm}
This naturally reduced to the Hausdorff-Young inequality \eqref{HYineq} when $b=p'$ and to the Paley inequality \eqref{Paleyin} when $b=p.$
\begin{proof}
From Theorem \ref{Paley}, the operator  defined by 
$$Af(\xi)= \mathcal{F}_{k, a} f(\xi),\,\,\,\,\xi \in \mathbb{R}^N$$
is bounded from $L^p_{k, a}(\mathbb{R}^N)$ to $L^{p}(\mathbb{R}^N,\omega_0  d\mu'),$ where $d\mu'(\xi)=d\mu_{k, a}(\xi)$ and $\omega_{0}(\xi)=  \psi(\xi)^{2-p}.$ From Theorem \ref{HY}, we deduce that $A:L^p_{k, a}(\mathbb{R}^N) \rightarrow L^{p'}(\mathbb{R}^N, \omega_1 d\mu')$ with $d\mu'(\xi)=d\mu_{k, a}(\xi)$ and   $\omega_1(\xi)= 1$  admits a bounded extension. By using the real interpolation (Theorem \ref{interpolationoperator} above) we will prove that $A:L^p_{k, a}(\mathbb{R}^N) \rightarrow L^{b}(\mathbb{R}^N, \omega d\mu'),$ $p\leq b\leq p',$ is bounded,
where the space $L^p(\mathbb{R}^N,\, \omega d\mu')$ is defined by the norm 
$$\|\sigma\|_{L^p(\mathbb{R}^N,\, \omega d\mu')}:=\left( \int_{ \mathbb{R}^N} |\sigma(\xi)|^p \omega(\xi) \,d\mu'(\xi) \right)^{\frac{1}{p}}= \left( \int_{ \mathbb{R}^N} |\sigma(\xi)|^p \omega(\xi) d\mu_{k, a}(\xi) \right)^{\frac{1}{p}}$$
 and $\omega(\xi)$ is positive function over $\mathbb{R}^N$ to be determined. To compute $\omega,$ we can use Theorem \ref{interpolationoperator}, by fixing $\theta\in (0,1)$ such that $\frac{1}{b}=\frac{1-\theta}{p}+\frac{\theta}{p'}$. In this case $\theta=\frac{p-b}{b(p-2)},$ and 
 \begin{equation}
     \omega= \omega_0^{\frac{p(1-\theta)}{p_0}} \omega_1^{\frac{p\theta}{p_1}}= \psi(\xi)^{1-\frac{b}{p'}}.     
 \end{equation}
 Thus we finish the proof.
\end{proof}

An  operator  $A$ is a  Fourier multiplier then there exists a measurable function $h :\R^N \rightarrow \mathbb{C},$ known as the symbol associated with $A,$ such that $$\mathcal{F}_{k, a}(Af)(\xi)=h(\xi) \mathcal{F}_{k,a}f(\xi),\,\,\,\, \xi \in \R^N,$$ for all $f$ belonging to a suitable function space on $\R^N.$  
	    In the next result, we show that if the symbol $h$ of a Fourier multipliers $A$ defined on $C_c^\infty(\R^N)$ satisfies certain H\"ormander's condition then $A$ can be extended as a bounded linear operator from $L^p_{k, a}(\R^N)$ to $L^q_{k, a}(\R^N)$ for the range $1<p \leq 2 \leq q <\infty.$

\begin{thm} \label{Jacobimult}  Let $1<p \leq 2 \leq q<\infty.$  Suppose that $A$ is a  Fourier multiplier with symbol $h.$ Then we have 
$$\|A\|_{L^p_{k, a}(\mathbb{R}^N) \rightarrow L^q_{k, a}(\mathbb{R}^N)}\lesssim \sup_{s>0} s \left[ \int_{\{ \xi \in \mathbb{R}^N: |h(\xi)|\geq s\}} d\mu_{k, a}(\xi) \right]^{\frac{1}{p}-\frac{1}{q}}.$$
   \end{thm}
\begin{proof} 
 Let us first assume that $p \leq q',$ where $\frac{1}{q}+\frac{1}{q'}=1.$ Since $q' \leq 2,$ the Hausdorff-Young inequality gives that 
 \begin{align*}
     \|Af\|_{L^q_{k, a}(\mathbb{R}^N)} \leq \|\mathcal{F}_{k, a}(Af)\|_{L^{q'}_{k, a}(\mathbb{R}^N)} = \|h \mathcal{F}_{k, a}f\|_{L^{q'}_{k, a}(\mathbb{R}^N)}
 \end{align*}
 
 The case $q' \leq (p')'=p$ can be reduced to the case $p \leq q'$ as follows. Using the duality of $L^p$-spaces we have $\|A\|_{L^p_{k, a}(\mathbb{R}^N) \rightarrow L^q(\mathbb{R}^N)}= \|A^*\|_{L^{q'}_{k, a}(\mathbb{R}^N) \rightarrow L^{p'}_{k, a}(\mathbb{R}^N)}.$ The symbol of adjoint operator $A^*$  is equal to $\check{h},$ which equal to  $h$ and obviously we have $|\check{h}|= |h|$ (see Theorem 4.2 in \cite{ARN1}). 
 Now, we are in a position to apply Theorem \ref{HYP}. Set $\frac{1}{p}-\frac{1}{q}=\frac{1}{r}.$ Now, by applying  Theorem \ref{HYP} with $\psi= |h|^r$ with $b=q'$ we get 
 $$\|h \mathcal{F}_{k, a}f\|_{L^{q'}(\mathbb{R}_+, Adx)} \lesssim \left(  \sup_{s>0} s \int_{\underset{|h(\xi)|^r > s}{\xi \in \mathbb{R}^N}} d\mu_{k, a}(\xi)   \right)^{\frac{1}{r}} \|f\|_{L^p_{k, a}(\mathbb{R}^N)} $$ for all $f \in L^p_{k, a}(\mathbb{R}^N),$ in view of $\frac{1}{p}-\frac{1}{q}=\frac{1}{q'}-\frac{1}{p'}=\frac{1}{r.}$ Thus, for $1<p \leq 2 \leq q<\infty,$ we obtain 
 
 $$\|Af\|_{L^q_{k, a}(\mathbb{R}^N)} \lesssim \left(  \sup_{s>0} s \int_{\underset{|h(\xi)|^r > s}{\xi \in \mathbb{R}^N}} d\mu_{k, a}(\xi)   \right)^{\frac{1}{r}} \|f\|_{L^p_{k, a}(\mathbb{R}^N)}.$$
 Further, the proof follows from the following inequality: 
 \begin{align*}
     \left(  \sup_{s>0} s \int_{\underset{|h(\xi)|^r > s}{\xi \in \mathbb{R}^N}} d\mu_{k, a}(\xi)  \right)^{\frac{1}{r}} &= \left(  \sup_{s>0} s \int_{\underset{|h(\xi)| > s^{\frac{1}{r}}}{\xi \in \mathbb{R}^N}} d\mu_{k, a}(\xi)    \right)^{\frac{1}{r}} \\&=  \left(  \sup_{s>0} s^{\frac{1}{r}} \int_{\underset{|h(\xi)| > s} {\xi \in \mathbb{R}^N}} d\mu_{k, a}(\xi)    \right)^{\frac{1}{r}} \\&= \sup_{s>0} s \left(   \int_{\underset{|h(\xi)| > s} {\xi \in \mathbb{R}^N}} d\mu_{k, a}(\xi)    \right)^{\frac{1}{r}},
 \end{align*} proving Theorem \ref{Jacobimult}.
\end{proof}

\begin{rem}
For $a=2$ and $k\equiv 0,$ we recover the classical theorem of H\"ormander \cite{Hormander1960} on $L^p$-$L^q$ boundedness of Fourier multipliers on $\mathbb{R}^N $ as in this case  $\mathcal{F}_{k,a}$ and $\mu_{k,a}$ become the Euclidean Fourier transform and the Lebesgue measure on $\mathbb{R}^N,$ respectively.   
\end{rem}

As an application of Theorem \ref{Jacobimult} we get the following result. 
\begin{cor}
 Let $0<\gamma< 2 \langle k \rangle+N+a-2$ and let $h$ be a measurable function on $\mathbb{R}^N$ such that
 $$|h(\xi)| \lesssim \|\xi \|^{-\gamma},$$ where $\|\xi\|$ is the Euclidean norm of $\xi \in \mathbb{R}^N.$   Then the $(k,a)$-Fourier multiplier $T_h$ with symbol $h$ is bounded from $L^p_{k,a}(\mathbb{R}^N)$ to $L^q_{k, a}(\mathbb{R}^N)$ provided that \begin{equation} \label{assu}
     1 < p \leq 2 \leq q < \infty,\,\, \frac{1}{p}-\frac{1}{q}=\frac{\gamma}{2 \langle k \rangle+N+a-2}.
 \end{equation}
\end{cor}
\begin{proof}
It follows from Theorem \ref{Jacobimult}  that 
\begin{align*}
    \|A\|_{L^p_{k, a}(\mathbb{R}^N) \rightarrow L^q_{k, a}(\mathbb{R}^N)} &\lesssim \sup_{s>0} s \left[  \int_{\{ \xi \in \mathbb{R}^N: |h(\xi)|\geq s\}} d\mu_{k, a}(\xi) \right]^{\frac{1}{p}-\frac{1}{q}} \\&\lesssim \sup_{s>0} s \left[ \int_{\{ \xi \in \mathbb{R}^N:\,\, s \lesssim \|\xi \|^{-\gamma} \}} d\mu_{k, a}(\xi) \right]^{\frac{1}{p}-\frac{1}{q}}.
\end{align*}
Now, using the polar coordinates on $\mathbb{R}^N$ and the fact that in polar coordinates it holds that $d\mu_{k,a}(x) (=v_{k,a}(x) \,dx):= r^{2\langle k \rangle+N+a-3} v_k(\theta)\, dr \,d\sigma(\theta)$ (see \cite{John}), we get
\begin{align*}
    \|A\|_{L^p_{k, a}(\mathbb{R}^N) \rightarrow L^q_{k, a}(\mathbb{R}^N)} &\lesssim \sup_{s>0} s \left[  \int_{\{ r \in \mathbb{R}_+:\,\, r \lesssim s^{-\frac{1}{\gamma}}  \}} r^{2 \langle k \rangle+N+a-3} dr \right]^{\frac{1}{p}-\frac{1}{q}} \\ & \lesssim \sup_{s>0} s \left[ s^{-\frac{2 \langle k \rangle+N+a-2}{\gamma}} \right]^{\left( \frac{1}{p}-\frac{1}{q} \right)} = \sup_{s>0} 1<\infty,
\end{align*} by using the assumption  \eqref{assu}.
\end{proof}

 \section{Applications to nonlinear PDEs}
 
 This section is devoted to the applications of our main result on $L^p$-$L^q$ boundedness of $(k, a)$-Fourier multipliers  to the well-posedness of abstract Cauchy problem on $\mathbb{R}^N$. The method we use here is the same as in \cite{CKNR} for the case of the  Fourier analysis associated to the biorthogonal eigenfunction expansion of a model operator on smooth manifolds having discrete spectrum.
 
 \subsection{Nonlinear Heat equation}
 Let us consider the following Cauchy problem of nonlinear evolution equation in the space $L^\infty(0, T, L^2_{k,a}(\mathbb{R}^N)),$
 \begin{equation} \label{heat}
     u_t-|Bu(t)|^p=0,\quad u(0)=u_0,
 \end{equation}
 where $B$ is a linear operator on $L^2_{k, a}(\mathbb{R}^N)$ and $1<p<\infty.$
 
 We say that the heat equation \eqref{heat} admits a solution $u$ if 
 \begin{equation}\label{heatsol}
     u(t)=u_0+\int_0^t |Bu(\tau)|^p\, d\tau
 \end{equation} in the space $L^\infty(0, T, L^p_{k, a}(\mathbb{R}^N))$ for every $T<\infty.$
 We say that $u$ is a local solution of \eqref{heat} if it satisfies the equation \eqref{heatsol} in the space $L^\infty(0, T^*, L^2_{k,a }(\mathbb{R}^N))$ for some $T^*>0.$
 \begin{thm}
 Let $1<p<\infty.$ Suppose that $B$ is Fourier multiplier such that its symbol $h$ satisfies $$\sup_{s>0} s \left[ \int_{\{ \xi \in \mathbb{R}^N: |h(\xi)|\geq s\}} d\mu_{k, a}(\xi) \right]^{\frac{1}{2}-\frac{1}{2p}}<\infty.$$ Then the Cauchy problem \eqref{heat} has a local solution in the space $L^\infty(0, T^*, L^2_{k,a}(\mathbb{R}^N))$ for some $T^*>0.$ 
 \end{thm}
 \begin{proof}
 By integrating equation \eqref{heat} w.r.t. $t$ one get
$$
u(t)=u_{0} + \int\limits_0^t |Bu(\tau)|^{p} d\tau.
$$
By taking the $L^2$-norm on both sides, one obtains
\begin{align*}
    \|u(t)\|_{L^{2}_{k,a}(\mathbb{R}^N)}^{2}  &\leq C \Bigg(\|u_0\|_{L^2_{k,a}(\mathbb{R}^N)}^2+\left\| \int_0^t |Bu(t)|^p\, d\tau \right\|^2_{L^2_{k,a}(\mathbb{R}^N)} \Bigg)\\& = C \Bigg(\|u_0\|_{L^2_{k,a}(\mathbb{R}^N)}^2+ \int_{\mathbb{R^N}} \left| \int_0^t |Bu(t)|^p\, d\tau  \right|^2 d\mu_{k,a}(x)\Bigg).
\end{align*}
Using the inequality $\int_0^t|Bu(\tau)|^p\, d\tau \leq (\int_0^t 1 \, d\tau)^{\frac{1}{2}} (\int_0^t|Bu(\tau)|^{2p}\, d\tau)^{\frac{1}{2}}= t^{\frac{1}{2}} (\int_0^t|Bu(\tau)|^{2p}\, d\tau)^{\frac{1}{2}},$ we get
\begin{align*}
    \|u(t)\|_{L^{2}_{k,a}(\mathbb{R}^N)}^{2}  & \leq  C \Bigg(\|u_0\|_{L^2_{k,a}(\mathbb{R}^N)}^2+ t \int_{\mathbb{R^N}}   \int_0^t |Bu(t)|^{2p}\, d\tau\,  d\mu_{k,a}(x)\Bigg)\\& \leq  C \Bigg(\|u_0\|_{L^2_{k,a}(\mathbb{R}^N)}^2+ t    \int_0^t \int_{\mathbb{R^N}} |Bu(t)|^{2p}\,   d\mu_{k,a}(x)\, d\tau\Bigg) \\& \leq  C \Bigg(\|u_0\|_{L^2_{k,a}(\mathbb{R}^N)}^2+ t    \int_0^t \|Bu(t)\|^{2p}_{L^{2p}_{k,a}(\mathbb{R}^N)}\, d\tau\Bigg).
\end{align*}

Next, using the condition on the symbol $h$ it can be seen, as an application of Theorem \ref{Jacobimult}, that the operator $B$ is a bounded operator from $L^2_{k,a }(\mathbb{R}^N)$ to $L^{2p}_{k,a}(\mathbb{R}^N)$, that is, $\|B u(t)\|_{L^{2p}_{k,a}(\mathbb{R}^N)} \leq C_1 \|u(t)\|_{L^{2}_{k,a}(\mathbb{R}^N)}$ and, therefore, the above inequality yields
\begin{align}\label{EQ:space-norm}
    \|u(t)\|_{L^{2}_{k,a}(\mathbb{R}^N)}^{2} \leq C \Bigg(\|u_0\|_{L^2_{k,a}(\mathbb{R}^N)}^2+ t    \int_0^t  \|u(t)\|^{2p}_{L^{2}_{k,a}(\mathbb{R}^N)}\, d\tau\Bigg),
\end{align}
for some constant $C$ independent from $u_0$ and $t$.

Finally, by taking $L^{\infty}$-norm in time on both sides of the estimate \eqref{EQ:space-norm}, one obtains
\begin{equation}\label{EQ:time-space-norm}
\|u(t)\|_{L^{\infty}(0, T; L^{2}_{k,a}(\mathbb{R}^N))}^{2}\leq C\Big(\|u_{0}\|_{L^{2}_{k,a}(\mathbb{R^N)}}^{2} + T^{2} \|u\|^{2p}_{L^{\infty}(0, T; L^{2}_{k, a}(\mathbb{R}^N))}\Big).
\end{equation}

Let us introduce the following set
\begin{equation}\label{u-Set}
S_c:=\left\{u\in L^{\infty}(0, T; L^{2}_{k, a}(\mathbb{R}^N)): \|u\|_{L^{\infty}(0, T; L^{2}_{k,a}(\mathbb{R^N))}} \leq c \|u_{0}\|_{L^{2}_{k,a}(\mathbb{R}^N)}\right\},
\end{equation}
for some constant $c \geq 1$. Then, for $u \in S_c$ we have
$$
\|u_{0}\|_{L^{2}_{k,a}(\mathbb{R}^N)}^{2} + T^{2} \|u\|^{2p}_{L^{\infty}(0, T; L^{2}_{k, a}(\mathbb{R}^N))} \leq 
\|u_{0}\|_{L^{2}_{k,a}(\mathbb{R^N)}}^{2} + T^{2} c^{2p} \|u_0\|^{2p}_{L^{2}_{k,a}(\mathbb{R}^N)}.
$$
Finally, for $u$ to be from the set $S_c$ it is enough to have, by invoking \eqref{EQ:time-space-norm}, that  
$$
\|u_{0}\|_{L^{2}_{k,a}(\mathbb{R}^N)}^{2} + T^{2} c^{2p} \|u_0\|^{2p}_{L^{2}_{k, a}(\mathbb{R}^N)}\leq c^{2} \|u_0\|_{L^{2}_{k,a}(\mathcal{R}^N)}^{2}.
$$
It can be obtained by requiring the following,
$$
T \leq T^{\ast}:=\frac{\sqrt{c^{2}-1}}{c^{p}\|u_0\|_{L^{2}_{k,a}(\mathbb{R}^N)}}.
$$
Thus, by applying the fixed point theorem, there exists a unique local solution $u\in L^{\infty}(0, T^{\ast}; L^{2}_{k,a }(\mathbb{R}^N))$ of the Cauchy problem \eqref{heat}.
 \end{proof}
 
 \subsection{Nonlinear Wave Equation} In this subsection, we will consider that the initial value problem  for the nonlinear wave equation  
\begin{align}\label{E-WNLE}
u_{tt}(t) - b(t)|Bu(t)|^{p} = 0,
\end{align} with the initial condition 
$$
u(0)=u_0, \,\,\, u_t(0)=u_1,
$$
where $b$ is a positive bounded function depending only on time, $B$ is a linear operator in $L^2_{k,a}(\mathbb{R}^N)$ and $1< p<\infty$. We intend to study the well-posedness of the wave equation \eqref{E-WNLE}.

We say that the initial value problem \eqref{E-WNLE} admits a global solution $u$ if it satisfies
\begin{equation}\label{E-WNLE-Sol}
u(t)=u_{0} + t u_{1} + \int\limits_0^t (t-\tau) b(\tau) |Bu(\tau)|^{p} d\tau
\end{equation}
in the space $L^{\infty}(0, T; L^{2}_{k,a}(\mathbb{R}^N))$ for every $T<\infty$.

We say that \eqref{E-WNLE} admits a local solution $u$ if it satisfies
the equation \eqref{E-WNLE-Sol} in the space $L^{\infty}(0, T^{\ast}; L^{2}_{k, a}(\mathbb{R}^N))$ for some $T^{\ast}>0$.

\begin{thm}\label{Th: E-WNLE}
Let $1< p<\infty$. Suppose that $B$ is a Fourier multiplier such that its symbol $h$ satisfies $$\sup_{s>0} s \left[ \int_{\{ \xi \in \mathbb{R}^N: |h(\xi)|\geq s\}} d\mu_{k, a}(\xi) \right]^{\frac{1}{2}-\frac{1}{2p}}<\infty.$$

\begin{itemize}
    \item [(i)] If $\|b\|_{L^{2}(0, T)}<\infty$ for some $T>0$ then the Cauchy problem \eqref{E-WNLE} has a local solution in  $L^{\infty}(0, T; L^{2}_{k,a}(\mathbb{R}^N))$.
    \item [(ii)] Suppose that $u_1$ is identically equal to zero. Let $\gamma>3/2$. Moreover, assume that $\|b\|_{L^{2}(0, T)}\leq c \, T^{-\gamma}$ for every $T>0$, where $c$ does not depend on $T$. Then, for every $T>0$, the Cauchy problem \eqref{E-WNLE} has a global solution in the space $L^{\infty}(0, T; L^{2}_{k,a}(\mathbb{R}^N))$ for sufficiently small $u_0$ in $L^2$-norm.
\end{itemize}
\end{thm}
\begin{proof} (i) By integrating the equation \eqref{E-WNLE} two times  in $t$ one get
$$
u(t)=u_{0} + t u_{1} + \int\limits_0^t (t-\tau) b(\tau) |Bu(\tau)|^{p} d\tau.
$$
By taking the $L^2$-norm on both sides, for $t<T$ one obtains by simple calculation that

\begin{align*}
    \|u(t)\|_{L^{2}_{k,a}(\mathbb{R}^N)}^{2}\leq & C\left\{ \|u_{0}\|_{L^{2}_{k,a}(\mathbb{R}^N)}^{2} + t^2 \|u_{1}\|_{L^{2}_{k,a}(\mathbb{R}^N)}^{2}+ \left\|\int\limits_0^t (t-\tau) b(\tau) |Bu(\tau)|^{p} d\tau \right\|_{L^{2}_{k,a}(\mathbb{R}^N)}^{2} \right\} \\& \leq C\left\{ \|u_{0}\|_{L^{2}_{k,a}(\mathbb{R}^N)}^{2} + t^2 \|u_{1}\|_{L^{2}_{k,a}(\mathbb{R}^N)}^{2}+ \int_{\mathbb{R}^N} \Big|\int\limits_0^t (t-\tau) b(\tau) |Bu(\tau)|^{p} d\tau \Big|^2 d\mu_{k,a}(x) \right\} \\& \leq C\left\{ \|u_{0}\|_{L^{2}_{k,a}(\mathbb{R}^N)}^{2} + t^2 \|u_{1}\|_{L^{2}_{k,a}(\mathbb{R}^N)}^{2}+ \int_{\mathbb{R}^N} \Big(t \int\limits_0^t \Big| b(\tau)|Bu(\tau)|^{p}\Big| d\tau \Big)^{2} d\mu_{k,a}(x) \right\} \\& \leq C\left\{ \|u_{0}\|_{L^{2}_{k,a}(\mathbb{R}^N)}^{2} + t^2 \|u_{1}\|_{L^{2}_{k,a}(\mathbb{R}^N)}^{2}+ \int_{\mathbb{R}^N} t^{2} \int\limits_0^t \Big| b(\tau)\Big|^{2} d\tau \int\limits_0^t \Big| Bu(\tau)\Big|^{2p} d\tau d\mu_{k,a}(x) \right\} \\& \leq C\left\{ \|u_{0}\|_{L^{2}_{k,a}(\mathbb{R}^N)}^{2} + t^2 \|u_{1}\|_{L^{2}_{k,a}(\mathbb{R}^N)}^{2}+ t^2 \|b\|_{L^2(0,T)} \int_{\mathbb{R}^N}  \int\limits_0^t \Big| Bu(\tau)\Big|^{2p} d\tau d\mu_{k,a}(x) \right\} \\& \leq C\left\{ \|u_{0}\|_{L^{2}_{k,a}(\mathbb{R}^N)}^{2} + t^2 \|u_{1}\|_{L^{2}_{k,a}(\mathbb{R}^N)}^{2}+ t^2 \|b\|_{L^2(0,T)} \int\limits_0^t \|Bu(\tau)\|^{2p}_{L^{2p}_{k,a}(\mathbb{R}^N)} d\tau \right\}.
\end{align*}
Next, using the condition on the symbol it can be seen, as an application of Theorem \ref{Jacobimult}, that the operator $B$ is a bounded operator from $L^2_{k,a }(\mathbb{R}^N)$ to $L^{2p}_{k,a}(\mathbb{R}^N)$, that is, $\|B u(t)\|_{L^{2p}_{k,a}(\mathbb{R}^N)} \leq C_1 \|u(t)\|_{L^{2}_{k,a}(\mathbb{R}^N)}$ and, therefore, the above inequality yields
\begin{equation}
\label{EQ: WE-space-norm}
\|u(t)\|_{L^{2}_{k,a}(\mathbb{R}^N)}^{2}\leq C(\|u_{0}\|_{L^{2}_{k,a}(\mathbb{R}^N)}^{2} + t^2 \|u_{1}\|_{L^{2}_{k,a}(\mathbb{R}^N)}^{2}+ t^{2} \|b\|_{L^{2}(0, T)}^{2} \int\limits_0^t \|u(\tau)\|^{2p}_{L^{2p}_{k,a}(\mathbb{R}^N)} d\tau),
\end{equation}
for some constant $C$ not depending on $u_0, u_1$ and $t$. Finally, by taking the $L^{\infty}$-norm in time on both sides of the estimate \eqref{EQ: WE-space-norm}, one obtains
\begin{equation}
\label{EQ: WE-time-space-norm}
\|u\|_{L^{\infty}(0, T; L^{2}_{k,a}(\mathbb{R}^N))}^{2}\leq C (\|u_{0}\|_{L^{2}_{k,a}(\mathbb{R}^N)}^{2} + T^2 \|u_{1}\|_{L^{2}_{k, a}(\mathbb{R}^N)}^{2}+ T^{3} \|b\|_{L^{2}(0, T)}^{2} \|u\|^{2p}_{L^{\infty}(0, T; L^{2}_{k,a}(\mathbb{R}^N))}). 
\end{equation}

Let us introduce the set
\begin{equation}
S_c:=\Big\{u\in L^{\infty}(0, T; L^{2}_{k,a}(\mathbb{R}^N)): \|u\|_{L^{\infty}(0, T; L^{2}_{k,a}(\mathbb{R}^N))}^{2} \leq
c(\|u_{0}\|_{L^{2}_{k,a}(\mathbb{R}^N)}^{2} + T^2 \|u_{1}\|_{L^{2}_{k,a}(\mathbb{R}^N)}^{2})\Big\}
\end{equation}
for some constant $c \geq 1$. Then, for $u \in S_c$ we have
\begin{align}\label{WE-Est}
&\|u_{0}\|_{L^{2}_{k,a }(\mathbb{R}^N)}^{2} + T^2 \|u_{1}\|_{L^{2}_{k,a}(\mathbb{R}^N)}^{2} + T^{3} \|b\|_{L^{2}(0, T)}^{2} \|u\|^{2p}_{L^{\infty}(0, T; L^{2}_{k,a}(\mathbb{R}^N))} \nonumber \\
&\leq \|u_{0}\|_{L^{2}_{k,a}(\mathbb{R}^N)}^{2} + T^2 \|u_{1}\|_{L^{2}_{k,a}(\mathbb{R}^N)}^{2} + T^{3} \|b\|_{L^{2}(0, T)}^{2} c^{p}\Big(\|u_{0}\|_{L^{2}_{k,a}(\mathbb{R}^N)}^{2} + T^2 \|u_{1}\|_{L^{2}_{k,a}(\mathbb{R}^N)}^{2}\Big)^{p}. 
\end{align}

Observe that, to be $u$ from the set $S_c$ it is enough to have, by invoking \eqref{EQ: WE-time-space-norm} and using \eqref{WE-Est}, that  
\begin{align*}
\begin{split}
\|u_{0}\|_{L^{2}_{k,a}(\mathbb{R}^N)}^{2} + &T^2 \|u_{1}\|_{L^{2}_{k,a}(\mathbb{R}^N)}^{2} + T^{3} \|b\|_{L^{2}(0, T)}^{2} c^{p}\Big(\|u_{0}\|_{L^{2}_{k,a}(\mathbb{R}^N)}^{2} + T^2 \|u_{1}\|_{L^{2}_{k,a}(\mathbb{R}^N)}^{2}\Big)^{p}\\
&\leq c(\|u_{0}\|_{L^{2}_{k,a}(\mathbb{R}^N)}^{2} + T^2 \|u_{1}\|_{L^{2}_{k,a}(\mathbb{R}^N)}^{2}).
\end{split}
\end{align*}
It can be obtained by requiring the following
$$
T \leq T^{\ast}:=\min\left[\left(\frac{c-1}{\|b\|_{L^{2}(0, T)}^{2}c^{p}\|u_0\|_{L^{2}_{k,a}(\mathbb{R}^N)}^{2p-2}}\right)^{\frac{1}{3}}, \, \left(\frac{c-1}{\|b\|_{L^{2}(0, T)}^{2}c^{p}\|u_1\|_{L^{2}_{k,a}(\mathbb{R}^N)}^{2p-2}}\right)^{\frac{1}{3}}\right].
$$
Thus, by applying the fixed point theorem, there exists a unique local solution $u\in L^{\infty}(0, T^{\ast}; L^{2}_{k,a}(\mathbb{R}^N))$ of the Cauchy problem \eqref{E-WNLE}.

To prove Part (ii), we  repeat the arguments of the proof of Part (i) to get \eqref{EQ: WE-time-space-norm}. Now, by taking into account assumptions on $u_1$ and $b$ inequality \eqref{EQ: WE-time-space-norm} yields
\begin{equation}
\label{EQ: WE-time-space-norm-2}
\|u\|_{L^{\infty}(0, T; L^{2}_{k,a}(\mathbb{R}^N))}^{2}\leq C \Big(\|u_{0}\|_{L^{2}_{k,a}(\mathbb{R}^N)}^{2} + T^{3-2\gamma}  \|u\|^{2p}_{L^{\infty}(0, T; L^{2}_{k,a}(\mathbb{R}^N))}\Big). 
\end{equation}

For a fixed constant $c \geq 1$, let us introduce the set
$$
S_c:=\Big\{u\in L^{\infty}(0, T; L^{2}_{k,a}(\mathbb{R}^N)): \|u\|_{L^{\infty}(0, T; L^{2}_{k,a}(\mathbb{R}^N))}^{2} \leq c T^{\gamma_{0}}\|u_{0}\|_{L^{2}_{k,a}(\mathbb{R}^N)}^{2}\Big\},
$$
with $\gamma_{0}>0$ is to be defined later. Now, note that for $u \in S_c$ we have
\begin{align*}
\|u_{0}\|_{L^{2}_{k,a}(\mathbb{R}^N)}^{2} + T^{3-2\gamma}  \|u\|^{2p}_{L^{\infty}(0, T; L^{2}_{k,a}(\mathbb{R}^N))} 
\leq \|u_{0}\|_{L^{2}_{k,a}(\mathbb{R}^N)}^{2} + T^{3-2\gamma+\gamma_{0}p} c^{p} \|u_{0}\|_{L^{2}_{k,a}(\mathbb{R}^N)}^{2p}.    
\end{align*}

To guarantee $u\in S_c$, by invoking \eqref{EQ: WE-time-space-norm-2} we require that
\begin{align*}
\|u_{0}\|_{L^{2}_{k,a}(\mathbb{R}^N)}^{2} + T^{3-2\gamma+\gamma_{0}p} c^{p} \|u_{0}\|_{L^{2}_{k,a}(\mathbb{R}^N)}^{2p} \leq c T^{\gamma_{0}} \|u_{0}\|_{L^{2}_{k,a}(\mathbb{R}^N)}^{2}.   
\end{align*}
Now by choosing $0<\gamma_0<\frac{2\gamma-3}{p}$ such that
$
\tilde{\gamma}:=3-2\gamma+\gamma_{0}p<0,
$ we obtain
$$
c^{p} \|u_{0}\|_{L^{2}_{k,a}(\mathbb{R}^N)}^{2p-2} \leq c T^{-\tilde{\gamma}+\gamma_{0}}.
$$
From the last estimate, we conclude that for any $T>0$ there exists sufficiently small $\|u_{0}\|_{L^{2}_{k,a}(\mathbb{R}^N)}$ such that IVP \eqref{E-WNLE} has a solution. It proves Part (ii) of Theorem \ref{Th: E-WNLE}.
\end{proof}

 \section*{Acknowledgment}
 The authors thank the referees for useful comments and suggestions which have greatly improved the exposition. The authors are grateful to Niyaz Tokmagambetov for fruitful discussions. VK also thanks Wentao Teng for helpful suggestions.
	VK and MR are supported by FWO Odysseus 1 grant G.0H94.18N: Analysis and Partial Differential Equations and by the Methusalem programme of the Ghent University Special Research Fund (BOF)
	(Grant number 01M01021). MR is also supported  by the EPSRC Grant EP/R003025/1 and by the FWO grant G022821N.

	\end{document}